\documentclass[12pt, twoside, english, headsepline]{article}

\usepackage[margin=30mm]{geometry}

\usepackage{scrlayer-scrpage}  
\clearscrheadfoot                 
\cehead[]{F. Cinque and E. Orsingher}        
\cohead[]{Analysis of fractional Cauchy problem and probabilistic applications}        
\cefoot[\pagemark]{\pagemark}  
\cofoot[\pagemark]{\pagemark}     

\usepackage{sectsty}
\sectionfont{\fontsize{14}{13}\selectfont}
\subsectionfont{\fontsize{12}{13}\selectfont}

\usepackage{tikz,graphicx}
\usepackage{amsmath,amsthm,amssymb,amsfonts, mathtools, amsbsy, dsfont}
\usepackage[utf8]{inputenx}
\usepackage{xcolor}
\usepackage{microtype}
\usepackage{hyperref}
\hypersetup{linkcolor=blue}

\newcommand*\dif{\mathop{}\!\mathrm{d}}  
\newcommand{\Conv}{\mathop{\scalebox{2.3}{\raisebox{-0.2ex}{$\ast$}}}}
\newcommand{\ceil}[1]{\left\lceil #1 \right\rceil}  
\newcommand*{\doi}[1]{\href{http://dx.doi.org/#1}{DOI: #1}}  
\newcommand*{\orcid}[1]{\href{https://orcid.org/#1}{ORCID: #1}}

\allowdisplaybreaks 

\newtheorem{theorem}{Theorem}[section] 
\newtheorem{proposition}{Proposition}[section] 
\newtheorem{lemma}[theorem]{Lemma}
\theoremstyle{definition} 
 
\newtheorem{example}{Example}[section]
\newtheorem{remark}{Remark}[section] 
\numberwithin{equation}{section} 

\providecommand{\keywords}[1]
{
  \small	
  \textbf{\textit{Keywords: }} #1
}
\providecommand{\MSC}[1]
{
  \small	
  \textit{2020 MSC: } #1   
}

\title{Analysis of fractional Cauchy problems with some probabilistic applications}
\author{Fabrizio Cinque$^1$ and Enzo Orsingher$^2$\\
        \small Department of Statistical Sciences, Sapienza University of Rome, Italy \\
        \small $^1$fabrizio.cinque@uniroma1.it, \orcid{0000-0002-9981-149X}\\ 
        \small$^2$enzo.orsingher@uniroma1.it
}

\begin{document}

\maketitle

\begin{abstract}
In this paper we give an explicit solution of Dzherbashyan-Caputo-fractional Cauchy problems related to equations with derivatives of order $\nu k$, for $k$ non-negative integer and $\nu>0$. The solution is obtained by connecting the differential equation with the roots of the characteristic polynomial and it is expressed in terms of Mittag-Leffler-type functions. Under the some stricter hypothesis the solution can be expressed as a linear combination of Mittag-Leffler functions with common fractional order $\nu$. We establish a probabilistic relationship between the solutions of differential problems with order $\nu/m$ and $\nu$, for natural $m$. Finally, we use the described method to solve fractional differential equations arising in the fractionalization of partial differential equations related to the probability law of planar random motions with finite velocities.
\end{abstract} \hspace{10pt}

\keywords{Dzherbashyan-Caputo derivative, Mittag-Leffler functions, Fourier transforms, Laplace transforms, Random motions.}

\MSC{Primary 34A08; Secondary 35R11, 60K99.}

\section{Introduction}

In this paper we consider fractional equations of the form
\begin{equation}\label{problemaIntroduzione}
\sum_{k=0}^N \lambda_k\frac{\partial^{\nu k} }{\partial t^{\nu k}}F(t, x) = 0,\ \ t\ge0,\ x\in \mathbb{R},\ \text{ with }\ \nu>0,
\end{equation}
where the roots of $\sum_{k=0}^N\lambda_k y^k = 0$ are different from $0$, and subject to the general initial conditions 
\begin{equation}\label{condizioniInizialiIntroduzione}
\frac{\partial^{l} F}{\partial t^{l}}\Big|_{t=0} = f_l(x),\ \ x\in \mathbb{R}^d,\ l=0,\dots,\ceil{N\nu} -1.
\end{equation}
The fractional derivatives are in the sense of Dzherbashyan-Caputo, that is, for $m\in \mathbb{N}_0$,
\begin{equation}\label{derivataCaputo}
\frac{\dif^\nu}{\dif t^\nu} f(t) = \begin{cases}
\begin{array}{l l}
\displaystyle\frac{1}{\Gamma(m-\nu)}\int_0^t (t-s)^{m-\nu-1}\frac{\dif^m}{\dif s^m}f(s)\dif s & \ \text{if}\ m-1<\nu<m\\[9pt]
\displaystyle\frac{\dif^m}{\dif t^m} f(t) &\ \text{if} \ \nu = m.
\end{array}\end{cases}
\end{equation}

We recall that the Laplace-transform of the fractional derivative of order $\nu>0$ can be expressed as, for suitable $\mu>0$,
\begin{equation}\label{trasformataLaplaceDerivataFrazionariaIntroduzione}
\int_0^\infty e^{-\mu t}\frac{\partial^\nu}{\partial t^\nu}f(t) \dif t= \mu^{\nu} \int_0^\infty e^{-\mu t}f(t) \dif t - \sum_{l=1}^{\ceil{\nu}} \mu^{\nu-l} \frac{\partial^{l-1}}{\partial t^{l-1}}f\Big|_{t=0}
\end{equation}
where we assume that $\lim_{t\longrightarrow \infty}  e^{-\mu t}\frac{\partial^{l-1}}{\partial t^{l-1}}f(t)=0,\ l\ge1$.

Dzherbashyan-Caputo fractional derivatives and the associated Cauchy problems have been intensively studied by many authors in the last decades, see for instance \cite{CM2008, KM2004, P1994} and the more recent papers such as \cite{K2016, M2019}. The main interest to such topic is arisen by their applications in several branches of science, such as physics and mechanics, see \cite{GKMR2014, P1999}. 

Fractional derivatives and the study of its related Cauchy problems also appear in the theory of stochastic processes. The main novelty of this work lies in the probabilistic relationship we establish between the solution of fractional Cauchy problems of different order and its application to the study of the fractional version of random motion with finite velocity.
\\

Our research aim is that of extending the results firstly presented in Orsingher and Beghin \cite{OB2004}, where the authors studied the time-fractional telegraph equation and the probabilistic interpretation of the solution. In particular, they were also able to prove that the probability law of the telegraph process subordinated with a reflecting Brownian motion satisfies the time-fractional differential equation
\begin{equation*}
\frac{\partial^{2\nu}u}{\partial t^{2\nu}} +2\lambda \frac{\partial^\nu u}{\partial t^\nu}  = c^2\frac{\partial^2u}{\partial x^2}, \ \ \ \text{with } \nu = \frac{1}{2},
\end{equation*}
subject to the initial condition $u(0,x)  =\delta(x)$  and $u_t(0,x)=0, \ x\in \mathbb{R}$. Later, these kinds of relationships were extended in a series of papers, see \cite{DoOT2014, OB2009}. In particular, in the paper by Orsingher and Toaldo \cite{OT2017} the authors studied the time-space-fractional equation
\begin{equation}\label{equazioneOrsingherToaldo}
\sum_{j=1}^m \lambda_j \frac{\partial^{\nu_j}u}{\partial t^{\nu_j}} = -c^2(-\Delta)^\beta,\ \ \ 0<\nu_j\le 1,\ \forall\ j,\ \beta \in (0,1],
\end{equation}
subject to the initial condition $u(0,x) = \delta(x),\ x\in \mathbb{R}^d$. In equation (\ref{equazioneOrsingherToaldo}), $-(-\Delta)^\beta$ denotes the fractional Laplacian (see \cite{K2017} for further details on this operator). The authors proved the relationship between this kind of equations and the probability law of an isotropic $d$-dimensional stable process, $S^{2\beta}$, subordinated with the inverse of a linear combination of independent stable processes, $L(t) = \inf\{s\ge0\,:\,\sum_{j=1}^m \lambda_j^{1/\nu_j} H_{\nu_j}(s)\ge t\}$, with $\lambda_j>0\ \forall\ j$ and $H_{\nu_j}$ stable processes of order $\nu_j\in(0,1)$.
\\

The novelty here is that the order of the Dzherbashyan-Caputo fractional derivatives appearing in (\ref{problemaIntroduzione}) can be arbitrarily large, although of the form $\nu k$. We point out that we state our main results in terms of ordinary fractional differential equations. Then, we are using this result to study partial fractional differential equations by means of the Fourier-transform approach.
\\

In Section 3, thanks to the use of the Laplace transform method, we show that the solution of the fractional Cauchy problem given by (\ref{problemaIntroduzione}) and (\ref{condizioniInizialiIntroduzione}) can be expressed as a combination of Mittag-Leffler-type functions with order of fractionality equal to $\nu>0$. Then we connect the solutions of problems with different \textit{order of fractionality} by means of a probability expectation such as, with $n\in\mathbb{N}$,
\begin{equation}\label{relazioneValoreAttesoIntroduzione}
F_{\nu/n}(t,x) = \mathbb{E}\,F_{\nu}\Biggl(\,\prod_{j=1}^{n-1}G_{j}^{(n)}(t),\,x\Biggr)
\end{equation}
where $F_{\nu/n}$ and $F_{\nu}$ are respectively the solution of a problem of degree $\nu/n$ and $\nu$ with suitable initial conditions and $G_{j}^{(n)}(t)$ are positive absolutely continuous random variables for each $t\ge0,\ j=1,\dots,n-1$ (see Section \ref{sottosezioneVariabiliAleatorieG} for details). 
The relationship (\ref{relazioneValoreAttesoIntroduzione}), where $F_{\nu/n}$ and $F_{\nu}$ are Fourier transforms of probability laws, leads to the equivalence (in terms of finite-dimensional distributions) of two processes, with the second one being time-changed through $\prod_{j=1}^{n-1}G_{j}^{(n)}(t)$.
\\

The problem we study in this paper was inspired by the fractionalization of the higher order partial differential equations governing the probability distribution of the position of random motion moving with a finite number of velocities. For instance, the fourth-order equation
\begin{equation}\label{equazioneQuartoOrdineIntro}
\Bigl(\frac{\partial^2 }{\partial t^2}+2\lambda\frac{\partial}{\partial t} +\lambda^2\Bigr)\biggl(\frac{\partial^2}{\partial t^2} +2\lambda\frac{\partial}{\partial t} -c^2\Bigl(\frac{\partial^2}{\partial x^2}+\frac{\partial^2}{\partial y^2}\Bigr)\biggr)p +c^4\frac{\partial^4 p}{\partial x^2\partial y^2} = 0,
\end{equation}
which emerges in the analysis of a planar stochastic dynamics with orthogonal-symmetrically chosen directions (see \cite{CO2023} for more details). The Fourier transforms of the equation (\ref{equazioneQuartoOrdineIntro}) has the form
\begin{equation}\label{equazioneQuartoOrdineFourierIntro}
\Bigl(\frac{\partial^2 }{\partial t^2}+2\lambda\frac{\partial}{\partial t} +\lambda^2\Bigr)\Bigl(\frac{\partial^2}{\partial t^2} +2\lambda\frac{\partial}{\partial t} +c^2(\alpha^2+\beta^2)\Bigr)F +c^4\alpha^2\beta^2F= 0,
\end{equation}
and its fractional version, with $\nu>0$, is
\begin{align}\label{equazioneQuartoOrdineFourierFrazionariaIntro}
\frac{\partial^{4\nu}F}{\partial t^{4\nu}} + 4\lambda \frac{\partial^{3\nu}F}{\partial t^{3\nu}}+5\lambda^2\frac{\partial^{2\nu}F}{\partial t^{2\nu}}+ 2\lambda^3\frac{\partial^\nu F}{\partial t^\nu} + c^2(\alpha^2+\beta^2)\Bigl(\frac{\partial^2 }{\partial t^2}+2\lambda\frac{\partial}{\partial t} +\lambda^2\Bigr)F+ c^4\alpha^2\beta^2 F= 0.
\end{align}
Equation (\ref{equazioneQuartoOrdineFourierFrazionariaIntro}) equivalently arises by considering the Fourier transform of the time-fractional version of equation (\ref{equazioneQuartoOrdineIntro}).
\\

In the last section of the paper we describe some applications of the theory constructed in Section 3 in the field of random motions with finite velocities. In detail, we study a method to derive the value of the (integer) derivatives of the Fourier transform (also called the characteristic function), in the time origin $t=0$, of the probability law of the position of the moving particle. Thanks to this result we can build the Cauchy problem solved by the characteristic function of a general motion and study its time-fractional counterpart. We provide two examples concerning planar random movements.

\section{Preliminary concepts}

\subsection{Convolutions of Mittag-Leffler-type functions}

The generalized Mittag-Leffler (GML) function, also known as three-parameter Mittag-Leffler fucntion, is a generalization of the exponential function. It has been first introduced by Prabhakar \cite{P1971} and is defined as
\begin{equation}\label{MittagLefflerGeneralizzata}
E_{\nu, \delta}^{\gamma} (x) = \sum_{k=0}^\infty \frac{\Gamma(\gamma+k)}{\Gamma(\gamma)\,k!}\frac{x^k}{ \Gamma(\nu k+ \delta)}, \ \ \ \ \ \nu, \gamma,\delta \in \mathbb{C}, Re(\nu), Re(\gamma), Re(\delta) >0, \ x\in \mathbb{R}.
\end{equation}
By considering $\gamma = 1$, (\ref{MittagLefflerGeneralizzata}) reduces to the well-known Mittag-Leffler function, see Pillai \cite{P1990}, Gorenflo \textit{et al.} \cite{GKMR2014}.
\\In this paper, as in many others, we are representing the solutions of fractional differential Cauchy problems in terms of Mittag-Leffler-type functions. These applications naturally appear in the fractional calculus, see Mainardi \cite{M2020}.

For our work it is useful to recall the Laplace transform of function (\ref{MittagLefflerGeneralizzata}),
\begin{equation}\label{inversaMittagLefferGeneralizzata}
\int_0^\infty e^{-\mu x}x^{\delta-1}E_{\nu,\delta}^{\gamma}(\beta x^{\nu}) \dif x = \frac{\mu^{\nu\gamma-\delta}}{(\mu^\nu-\beta)^\gamma}, \ \ \ \Big|\frac{\mu^\nu}{\beta}\Big|<1.
\end{equation}

Let $M \in \mathbb{N}$. Below we use the following multivariate analogue of the generalized Mittag-Leffler
\begin{equation}\label{GenMittagLefflerMulti}
E_{\nu,\delta}^{\gamma}(x) = \sum_{k_1,\dots,k_M=0}^\infty\,\prod_{j=1}^M\,\frac{\Gamma(\gamma_j+k_j)}{\Gamma(\gamma_j)\,k_j!}\,x_j^{k_j}\,\frac{1}{\Gamma\bigl(\nu\sum_{j=1}^Mk_j+\delta\bigr)},
\end{equation}
where $\gamma = (\gamma_1,\dots,\gamma_M)\in \mathbb{C}^M,\ \nu,\delta \in \mathbb{C}$, with $Re(\gamma_1),\dots,Re(\gamma_M),Re(\nu)>0$, and $x\in \mathbb{C}^M$. Function (\ref{GenMittagLefflerMulti}) is a particular case of the multivariate Mittag-Leffler introduced by Saxena \textit{et al.} \cite{SKRs2011} and used in Cinque \cite{C2022} to represent the distribution of the sum of independent generalized Mittag-Leffler random variables.

\begin{lemma}\label{lemmaConvMitLefGen}
Let $M \in \mathbb{N}$ and $t\ge0$. Also assume that $\gamma_1,\dots,\gamma_M\in \mathbb{C},\ \nu,\delta \in \mathbb{C}\setminus\{0\}$ such that $Re(\gamma_1),\dots,Re(\gamma_M),Re(\nu)>0$and $\eta_1\not=\dots\not = \eta_M\in\mathbb{C}$. Then,
\begin{equation}
\Biggl(\Conv_{j=1}^M x^{\delta_j-1}E_{\nu,\delta_j}^{\gamma_j}(\eta_j x^\nu)\Biggr) (t) = t^{\sum_{j=1}^M \delta_j-1} E_{\nu,\sum_{j=1}^M \delta_j}^{(\gamma_1,\dots,\gamma_M)}\Big(\eta_1t^\nu,\dots,\eta_Mt^\nu\Big)
\end{equation}
where the convolution is performed with respect to the variable $x\ge0$.
\end{lemma}

Note that the parameters $\delta_j$ appear just in terms of their summation $\sum_{j=1}^M\delta_j$, meaning that it does not matter how they are distributed among the factors of the convolution.

\begin{proof}
It is sufficient to show that for $n\in \mathbb{N}$ and suitable $\nu,\delta_0,\delta,\gamma_0,\dots,\gamma_n, \eta_0,\dots,\eta_n$,
$$ \Biggl(x^{\delta_0-1}E_{\nu,\delta_0}^{\gamma_0}(\eta_0x^\nu) \ast E_{\nu,\delta}^{(\gamma_1,\dots,\gamma_n)} (\eta_1x^\nu,\dots,\eta_n x^\nu)\Biggr)(t)= t^{\delta_0+\delta-1}E_{\nu,\delta_0+\delta}^{(\gamma_0,\gamma_1,\dots,\gamma_n)} (\eta_0 t^\nu,\eta_1t^\nu,\dots,\eta_n t^\nu).$$
Indeed,
\begin{align*}
\Biggl(&x^{\delta_0-1}E_{\nu,\delta_0}^{\gamma_0}(\eta_0x^\nu) \ast E_{\nu,\delta}^{(\gamma_1,\dots,\gamma_n)} (\eta_1x^\nu,\dots,\eta_n x^\nu)\Biggr)(t)\\
&=\sum_{k_0=0}^\infty \frac{\Gamma(\gamma_0+k_0)}{\Gamma(\gamma_0)\,k_0!} \eta_0^{k_0} \sum_{k_1,\dots,k_n = 0}^\infty\, \Biggl(\,\prod_{j=1}^n\,\frac{\Gamma(\gamma_j+k_j)}{\Gamma(\gamma_j)\,k_j!} \eta_j^{k_j}\Biggr) \int_0^t \frac{(t-x)^{\nu k_0+\delta_0-1}x^{\nu\sum_{j=1}^n k_j + \delta -1}}{\Gamma\bigl(\nu k_0+\delta_0\bigr)\Gamma\bigl(\nu\sum_{j=1}^n k_j+\delta\bigr)}\dif x\\
& =\sum_{k_0=0}^\infty \frac{\Gamma(\gamma_0+k_0)}{\Gamma(\gamma_0)\,k_0!} \eta_0^{k_0} \sum_{k_1,\dots,k_n = 0}^\infty\, \Biggl(\,\prod_{j=1}^n\,\frac{\Gamma(\gamma_j+k_j)}{\Gamma(\gamma_j)\,k_j!} \eta_j^{k_j}\Biggr) \frac{t^{\nu\sum_{j=0}^n k_j+\delta_0+\delta-1}}{\Gamma\bigl(\nu\sum_{j=0}^n k_j+\delta_0+\delta\bigr)}\\
& = \sum_{k_0,\dots,k_n = 0}^\infty\, \Biggl(\,\prod_{j=0}^n\,\frac{\Gamma(\gamma_j+k_j)}{\Gamma(\gamma_j)\,k_j!} \Bigl(\eta_j t^\nu\Bigr)^{k_j}\Biggr) \frac{t^{\delta_0+\delta-1}}{\Gamma\bigl(\nu\sum_{j=0}^n k_j+\delta_0+\delta\bigr)}.
\end{align*}

\end{proof}

For the convolution of $M$ two-parameters Mittag-Leffler functions we can derive an expression in terms of a linear combination of $M$ two-parameters Mittag-Leffler functions having all the same parameters.
\begin{proposition}\label{proposizioneConvoluzioneMittagLeffler}
Let $M \in \mathbb{N}$ and $t\ge0$. Also assume that $\gamma_1,\dots,\gamma_M\in \mathbb{C},\ \nu,\delta \in \mathbb{C}\setminus\{0\}$ such that $Re(\gamma_1),\dots,Re(\gamma_M),Re(\nu)>0$ and $\eta_1\not=\dots\not = \eta_M\in\mathbb{C}$. Then,
\begin{equation}\label{convoluzioneMittagLeffler}
\Biggl(\Conv_{i=1}^M x^{\delta_i-1}E_{\nu,\delta_i}(\eta_i x^\nu)\Biggr)(t) = t^{\sum_{h=1}^M \delta_h-1}\sum_{i=1}^M\frac{\eta_i^{M-1}}{\displaystyle\prod_{\substack{j=1\\j\not=i}}^M(\eta_i-\eta_j)} E_{\nu,\,\sum_{h=1}^M \delta_h}\bigl(\eta_it^\nu\bigr)
\end{equation}
where the convolution is performed with respect to the non-negative variable $x\ge0$.
\end{proposition}

\begin{proof}
First we recall that for $n,M\in\mathbb{N}_0$ and $\eta_1\not=\dots\not=\eta_N\in \mathbb{C}\setminus\{0\}$,
\begin{equation}\label{formulaPerCasoMolteplicitaUnitarie}
\sum_{i=1}^M \frac{\eta_i^{n}}{\prod_{\substack{j=1\\j\not=i}}^M(\eta_i-\eta_j)} =0, \ \ \text{with }\ n\le M-2.
\end{equation}
Then, we also note that the right-hand side of formula (\ref{convoluzioneMittagLeffler}) can be also written as
\begin{equation}\label{convoluzioneMittagLefflerAlternativa}
t^{\sum_{h=1}^M \delta_h-1}\sum_{i=1}^M\frac{\eta_i^{M-1}}{\prod_{\substack{j=1\\j\not=i}}^M(\eta_i-\eta_j)} E_{\nu,\,\sum_{h=1}^M \delta_h}\bigl(\eta_it^\nu\bigr) = \sum_{k=0}^\infty \frac{ t^{\nu k + \sum_{h=1}^M \delta_h -1}}{\Gamma\bigl(\nu k + \sum_{h=1}^M \delta_h\bigr)}\sum_{i=1}^M \frac{\eta_i^{k+M-1}}{\prod_{\substack{j=1\\j\not=i}}^M(\eta_i-\eta_j) }.
\end{equation}
We now proceed by induction. The induction base (M=2) can be found in Orsingher and Beghin \cite{OB2004}. Now, assume that (\ref{convoluzioneMittagLeffler}) holds for $M-1$.
\begin{align}
\Biggl(&\Conv_{i=1}^M x^{\delta_i-1}E_{\nu,\delta_i}(\eta_i x^\nu)\Biggr)(t) \nonumber\\
& = \sum_{i=1}^{M-1} \frac{\eta_i^{M-2}}{\prod_{\substack{j=1\\j\not=i}}^{M-1}(\eta_i-\eta_j) } \int_0^t x^{\delta_M-1} E_{\nu, \delta_M}\bigl(\eta_Mx^\nu\bigr) (t-x)^{ \sum_{h=1}^{M-1} \delta_h-1}E_{\nu,  \sum_{h=1}^{M-1} \delta_h}\Bigl(\eta_i(t-x)^\nu\Bigr)\dif x\label{usoBaseInduttiva}\\
&= \sum_{i=1}^{M-1} \frac{\eta_i^{M-2}}{\prod_{\substack{j=1\\j\not=i}}^{M-1}(\eta_i-\eta_j) } \sum_{k=0}^\infty \frac{ t^{\nu k + \sum_{h=1}^M \delta_h -1}}{\Gamma\bigl(\nu k + \sum_{h=1}^M \delta_h\bigr)}\Bigl(\frac{\eta_i^{k+1}}{\eta_i-\eta_M}+\frac{\eta_M^{k+1}}{\eta_M-\eta_i}\Bigr)\nonumber\\
& = \sum_{i=1}^{M-1}\, \frac{\eta_i^{M-1}\, t^{\sum_{h=1}^M \delta_h-1}}{\prod_{\substack{j=1\\j\not=i}}^{M}(\eta_i-\eta_j)} \,E_{\nu,\,\sum_{h=1}^M \delta_h}\bigl(\eta_it^\nu\bigr)\nonumber\\
&\ \ \  - t^{\sum_{h=1}^M \delta_h-1}\, E_{\nu,\,\sum_{h=1}^M \delta_h}\bigl(\eta_Mt^\nu\bigr) \,\eta_M \,\sum_{i=1}^{M-1}\, \frac{\eta_i^{M-2}}{\prod_{\substack{j=1\\j\not=i}}^{M}(\eta_i-\eta_j)} \label{ultimoPassaggioConvoluzioneMittagLeffler}\\
& =  \sum_{i=1}^{M-1}\, \frac{\eta_i^{M-1}\, t^{\sum_{h=1}^M \delta_h-1}}{\prod_{\substack{j=1\\j\not=i}}^{M}(\eta_i-\eta_j)} \,E_{\nu,\,\sum_{h=1}^M \delta_h}\bigl(\eta_it^\nu\bigr) + t^{\sum_{h=1}^M \delta_h-1}\, E_{\nu,\,\sum_{h=1}^M \delta_h}\bigl(\eta_Mt^\nu\bigr)\frac{\eta_M^{M-1}}{\prod_{\substack{j=1\\j\not=i}}^{M}(\eta_i-\eta_j)} .\nonumber
\end{align}
where in step (\ref{usoBaseInduttiva}) we used the induction base (i.e. with $M=2$) written as in (\ref{convoluzioneMittagLefflerAlternativa}), and in step (\ref{ultimoPassaggioConvoluzioneMittagLeffler}) we suitably used formula (\ref{formulaPerCasoMolteplicitaUnitarie}).
\end{proof}

\subsection{Generalization of absolute normal distribution}\label{sottosezioneVariabiliAleatorieG}
In \cite{BO2003} the authors introduced the following absolutely continuous positively distributed random variables. Let $n\in \mathbb{N}$ and $y>0$,
\begin{equation}\label{leggeSingolaG}
P\{G_j^{(n)}(t)\in \dif y\} = \frac{y^{j-1}\dif y}{n^{\frac{j}{n-1}-1}t^{\frac{j}{n(n-1)}}\Gamma(j/n)}e^{-\frac{y^n}{(n^nt)^{\frac{1}{n-1}}}}, \ \ \ t>0,\ j =1,\ \dots, n-1.
\end{equation}

Note that in the case of $n=2$ we have only one element and $G^{(2)}_1(t) = |B(2t)|,\ t\ge0,$ with $B$ being a standard Brownian motion.

If $G_1^{(n)}, \dots, G_{n-1}^{(n)}$ are independent, then the joint density reads, with $y_1,\dots,y_{n-1}>0$,
\begin{equation}
P\Bigg\{\bigcap_{j=1}^{n-1} \big\{G_j^{(n)}(t)\in \dif y_j\big\}\Bigg\} =  \Bigl(\frac{n}{2\pi}\Bigr)^{\frac{n-1}{2}} \frac{1}{\sqrt{t}} \Biggl(\prod_{j=1}^{n-1}  y_j^{j-1}\,\dif  y_j\Biggr)\,e^{-( n^nt)^{\frac{-1}{n-1}}\sum_{j=1}^{n-1}y_j^n}.
\end{equation}

Let $t>0$ and $n\ge 2$. It is easy to derive that the Mellin-transform of distribution (\ref{leggeSingolaG}) reads, for $s>0$,
\begin{equation*}
\int_0^{\infty} y^{s-1} f_{G_j^{(n)} (t)}(y)\dif y = \Bigl(nt^{1/n}\Bigr)^{\frac{s-1}{n-1}}\frac{\Gamma\bigl(\frac{s+j-1}{n}\bigr)}{\Gamma\bigl(\frac{j}{n}\bigr)}, \ \ \ j=1,\dots, n-1.
\end{equation*}
In the independence case, the Mellin-transform of the density, $f_{G^{(n)}(t)}$, of the product $G^{(n)} (t) = \prod_{j=1}^{n-1} G_{j}^{(n)}(t)$, is, with $s>0$,
\begin{equation*}
\int_0^{\infty} y^{s-1} f_{G^{(n)} (t)}(y)\dif y =\prod_{j=1}^{n-1} \Bigl(nt^{1/n}\Bigr)^{\frac{s-1}{n-1}}\frac{\Gamma\bigl(\frac{s+j-1}{n}\bigr)}{\Gamma\bigl(\frac{j}{n}\bigr)} = \frac{t^{\frac{s-1}{n}}}{\Gamma\bigl(\frac{s-1}{n}+1\bigr)}\Gamma(s).
\end{equation*}
where in the last equality we used the following $n$-multiplication formula of Gamma function for $z = 1/n$ and $s/n$,
$$ \prod_{j=1}^{n-1}\Gamma\Bigl(z+\frac{j-1}{n}\Bigr)= \frac{(2\pi)^{\frac{n-1}{2}}n^{\frac{1}{2}-nz}\,\Gamma(nz)}{\Gamma\Bigl(z+\frac{n-1}{n}\Bigr)}.$$

\section{Fractional differential Cauchy problem}

In this section we derive an explicit formula for the solution to the fractional Cauchy problem given by (\ref{problemaIntroduzione}) and (\ref{condizioniInizialiIntroduzione}). Hereafter we are considering functions $f:[0,\infty)\times\mathbb{R}^d\longrightarrow \mathbb{R}$ such that $\lim_{t\longrightarrow \infty}  e^{-\mu t}\frac{\partial^{l-1}}{\partial t^{l-1}}f(t)=0\ \forall\ l$.

\begin{theorem}\label{teoremaGenerale}
Let $d,N\in \mathbb{N},\ \nu>0$ and $\lambda_0,\dots,\lambda_N\in \mathbb{R}$. If 
\begin{equation}\label{rappresentazionePolinomio}
\sum_{k=0}^N \lambda_k x^k = \prod_{j=1}^M(x-\eta_j)^{m_j} \ \ \text{with }\ \eta_1,\dots,\eta_M\in \mathbb{C}\setminus \{0\},
\end{equation}
then, the solution to the fractional Cauchy problem of parameter $\nu$
\begin{equation}\label{problemaGenerale}
\begin{cases}
\displaystyle\sum_{k=0}^N \lambda_k\frac{\partial^{\nu k} }{\partial t^{\nu k}}F(t, x) = 0,\ \ t\ge0,\ x\in \mathbb{R}^d\\[15pt]
\displaystyle\frac{\partial^{l} F}{\partial t^{l}}\Big|_{t=0} = f_l(x),\ \ x\in \mathbb{R}^d,\ l=0,\dots,\ceil{\nu N} -1,
\end{cases}
\end{equation}
is the function $F:[0,\infty)\times\mathbb{R}^d\longrightarrow \mathbb{R}$ given by
\begin{equation}\label{tesiGenerale}
F(t,x) = \sum_{l=0}^{\ceil{\nu N}-1} f_{l}(x) \sum_{k=k_l}^N \lambda_k\, t^{\nu(N-k)+l} \, E_{\nu, \,\nu(N-k)+l+1}^{(m_1,\dots,m_M)}\Big(\eta_1t^\nu, \dots, \eta_Mt^\nu\Big),
\end{equation}
with $k_l = \min\{k=1,\dots,N\,:\, \nu k> l\},\ l=0,\dots,\ceil{\nu N}-1$.
\end{theorem}
Note that $k_0=1$ and $l-1< \nu k \le l$ for all $k_{l-1}\le k< k_l$. Formula (\ref{tesiGenerale}) can be also written inverting the sums into $ \sum_{k=1}^N \sum_{l=0}^{\ceil{\nu k}-1}$.

Condition (\ref{rappresentazionePolinomio}) implies that $\eta_1,\dots,\eta_M$ are the $M$ roots of the $N$-th order polynomial with coefficients $\lambda_0,\dots,\lambda_N$, respectively with algebraic molteplicity $m_1,\dots,m_M\ge1$.
In the case $M=N$, all the roots have algebraic molteplicity equal to 1 and the solution can be expressed in terms of a combination of Mittag-Leffler functions (see Theorem \ref{teoremaMolteplicitaUnitarie}).

\begin{proof}
By means of the $t$-Laplace transform, the differential equation in problem (\ref{problemaGenerale}) turns into, for $\mu\ge0$ (we use the notation $G(\mu,x) =\mathcal{L}(F)(\mu,x)=\int_0^\infty e^{-\mu t}F(t,x)\dif t$ and keep in mind formula (\ref{trasformataLaplaceDerivataFrazionariaIntroduzione}))
\begin{align}
0&= \mathcal{L}\Biggl(\sum_{k=0}^N \lambda_k\frac{\partial^{\nu k} }{\partial t^{\nu k}}F \Biggr) = \sum_{k=0}^N \lambda_k\,\mathcal{L}\Bigl(\frac{\partial^{\nu k} }{\partial t^{\nu k}}F\Bigr) =\lambda_0 G +\sum_{k=1}^{N}  \lambda_k \Bigl[\mu^{\nu k}G-\sum_{l=1}^{\ceil{\nu k}}\mu^{\nu k -l} f_{l-1}\Bigr], \nonumber
\end{align}
which gives
\begin{align}   \label{trasformataLaplaceGenerale}
G(\mu,x) &= \frac{\displaystyle \sum_{k=1}^{N} \lambda_k \sum_{l=1}^{\ceil{\nu k}}\mu^{\nu k -l} f_{l-1}(x)}{\displaystyle\sum_{k=0}^N \lambda_k \mu^{\nu k}} = \frac{\displaystyle\sum_{l=1}^{\ceil{\nu N}} f_{l-1}(x)  \sum_{k=k_{l-1}}^N \lambda_k\, \mu^{\nu k -l} }{\displaystyle\prod_{h=1}^M \bigl(\mu^{\nu}-\eta_h\bigr)^{m_h}},
\end{align}
where we used hypothesis (\ref{rappresentazionePolinomio}) and $k_{l-1}$ is defined in the statement.

We now compute the $\mu$-Laplace inverse of the functions $\mu^{\nu k - l}/ \prod_{h=1}^M \bigl(\mu^{\nu}-\eta_h\bigr)^{m_h}$, for $l=1,\dots,\ceil{N\nu}$ and $k = k_{l-1},\dots,N$, by properly applying formula (\ref{inversaMittagLefferGeneralizzata}).
Let us consider $M_k = \min\{n\,:\, \sum_{h=1}^n m_h \ge k\}$, therefore $M_1 = 1$ because $m_1\ge1$ and $M_N = M$. Clearly, $\sum_{h=1}^{M_k} m_h \ge k> \sum_{h=1}^{M_{k}-1}m_h$ and $m_{M_k} \ge k -\sum_{h=1}^{M_{k}-1}m_h$; clearly $\sum_{h=1}^M m_h = N$.
We can decompose $\nu k -l$ as follows (it is not the only way):
\begin{align}
\nu k - l&=\nu\Bigl(\sum_{h=1}^{M_k-1}m_h + k - \sum_{h=1}^{M_k-1}m_h\Bigr) - l\,\frac{\sum_{h=1}^M m_h}{N} \nonumber\\
&= \nu\sum_{h=1}^{M_k-1}m_h +\nu\Bigl( k - \sum_{h=1}^{M_k-1}m_h\pm \sum_{h=M_k}^{M}m_h\Bigr) - l\,\frac{\sum_{h=1}^M m_h}{N}\nonumber \\
&=\sum_{h=1}^{M_k-1}\Bigl(\nu m_h-l\frac{m_h}{N}\Bigr) + \Biggl[\nu m_{M_k}-\Bigl(\nu m_{M_k} -\nu k + \nu \sum_{h=1}^{M_k-1}m_h + l\frac{m_{M_k}}{N}\Bigr)\Biggr] \nonumber\\
&\ \ \ + \sum_{h=M_k+1}^{M}\Biggl[ \nu m_h -\Bigl(\nu m_h+l\frac{m_h}{N}\Bigr)\Biggr].\label{decomposizioneEsponente}
\end{align}
In view of (\ref{decomposizioneEsponente}) we can write, by denoting with $\mathcal{L}^{-1}$ the inverse $\mu$-Laplace transform operator, for $l=1,\dots,\ceil{N\nu}$ and $k = k_{l-1},\dots,N$
\begin{align}
\mathcal{L}^{-1}& \Biggl(\frac{\mu^{\nu k -l} }{\displaystyle\prod_{h=1}^M \bigl(\mu^{\nu}-\eta_h\bigr)^{m_h}}\Biggr)(t) \nonumber\\
&=\prod_{h=1}^{M_k-1} \mathcal{L}^{-1}\Biggl( \frac{\mu^{\nu m_h-l m_h/N}}{\bigl(\mu^{\nu}-\eta_h\bigr)^{m_h}}\Biggr)(t)\, \mathcal{L}^{-1}\Biggl(  \frac{\mu^{\nu m_{M_k}- \bigl(\nu \sum_{h=1}^{M_k}m_h -\nu k + l m_{M_k}/N\bigr)} }{\bigl(\mu^{\nu}-\eta_h\bigr)^{m_{M_k}}}\Biggr)(t) \nonumber\\
& \ \ \ \times \prod_{h=M_k+1}^M \mathcal{L}^{-1}\Biggl(  \frac{\mu^{- l m_h/N}}{\bigl(\mu^{\nu}-\eta_h\bigr)^{m_h}} \Biggr)(t)\label{passaggioAntitrasformazioneLaplace}\\
& = \Conv_{h=1}^{M_k-1} t^{l m_h/N -1} E_{\nu,\,l m_h/N}^{m_h}\bigl(\eta_ht^\nu\bigr)\ast t^{\nu \sum_{h=1}^{M_k}m_h -\nu k +l m_{M_k}/N-1} E_{\nu,\,\nu \sum_{h=1}^{M_k}m_h -\nu k + l m_{M_k}/N}^{m_{M_k}}\bigl(\eta_{M_k}t^\nu\bigr)\nonumber\\
&\ \ \ \ast\Conv_{h=1}^{M_k-1} t^{\nu m_h - l m_h/N -1} E_{\nu,\,\nu m_h+ lm_h/N}^{m_h}\bigl(\eta_ht^\nu\bigr)\nonumber\\
& = t^{\nu (N-k)+l-1}E_{\nu,\,\nu (N-k)+l}^{(m_1,\dots,m_M)}\bigl(\eta_1t^\nu,\dots,\eta_Mt^\nu\bigr),\label{inversioneLaplaceGenerale}
\end{align}
where in step (\ref{passaggioAntitrasformazioneLaplace}) we used (\ref{inversaMittagLefferGeneralizzata}) and in the last step 
we used Lemma \ref{lemmaConvMitLefGen}. Note that in step (\ref{passaggioAntitrasformazioneLaplace}) it is necessary to keep the ``$\delta$'' terms greater than $0$ (see (\ref{inversaMittagLefferGeneralizzata})) and this is the main reason of using the above decomposition of $\nu k -l$.
\\By combining (\ref{trasformataLaplaceGenerale}) and (\ref{inversioneLaplaceGenerale}) we readily obtain result (\ref{tesiGenerale}) (after the change of variable $l' = l-1$).
\end{proof}

\begin{remark}[Non-homogeneous equation]
Under the hypothesis of Theorem \ref{teoremaGenerale} we can easily study the Cauchy problem in the case of a non-homogeneous fractional equation. In details, for $g:\mathbb{R}\times\mathbb{R}^d\longrightarrow\mathbb{R}$, such that there exists the $t$-Laplace transform, the solution of
\begin{equation*}
\begin{cases}
\displaystyle\sum_{k=0}^N \lambda_k\frac{\partial^{\nu k} }{\partial t^{\nu k}}F(t, x) = g(t,x),\ \ t\ge0,\ x\in \mathbb{R}^d\\[15pt]
\displaystyle\frac{\partial^{l} F}{\partial t^{l}}\Big|_{t=0} = f_l(x),\ \ x\in \mathbb{R}^d,\ l=0,\dots,\ceil{\nu N} -1,
\end{cases}
\end{equation*}
reads
\begin{align}
&F(t,x) \label{tesiGeneraleNonOmogenea}\\
& =  \sum_{l=0}^{\ceil{N\nu}-1} f_{l}(x) \sum_{k=k_l}^N \lambda_k\, t^{\nu(N-k)+l} \, E_{\nu, \,\nu(N-k)+l+1}^{(m_1,\dots,m_M)}\big(\eta t^\nu\big)- \int_0^t g(t-y,x)\,y^{\nu N-1}E_{\nu, \nu N}^{(m_1,\dots,m_M)}\big(\eta y^\nu\big)\dif y,\nonumber
\end{align}
where $\eta = (\eta_1,\dots,\eta_M)$.

The above results easily follows by observing that formula (\ref{trasformataLaplaceGenerale}) becomes 
\begin{equation*}\label{trasformataLaplaceGeneraleNonOmogenea}
G(\mu,x)  = \frac{\displaystyle\sum_{l=1}^{\ceil{N\nu}} f_{l-1}(x)  \sum_{k=k_{l-1}}^N \lambda_k\, \mu^{\nu k -l} -\mathcal{L}(g)(\mu,x)}{\displaystyle\prod_{h=1}^M \bigl(\mu^{\nu}-\eta_h\bigr)^{m_h}}
\end{equation*}
and we observe that the $\mu$-Laplace inverse of the term concerning the function $g$ is
\begin{align*}
\mathcal{L}^{-1}\Biggl(\mathcal{L}(g)(\mu,x)\Bigl(\prod_{h=1}^M \bigl(\mu^{\nu}-\eta_h\bigr)^{m_h}\Bigr)^{-1} \Biggr)(t,x) &= \int_0^t g(t-y,x)\,y^{\nu N-1}E_{\nu, \nu N}^{(m_1,\dots,m_M)}\big(\eta y^\nu\big)\dif y\\
\end{align*}
 where we used that $\mathcal{L}^{-1}\Bigl(\Bigl(\prod_{h=1}^M \bigl(\mu^{\nu}-\eta_h\bigr)^{m_h}\Bigr)^{-1} \Bigr)(t,x) = t^{\nu N-1}E_{\nu, \nu N }^{(m_1,\dots,m_M)}\big(\eta t^\nu\big)$ (obtained by proceeding as shown for (\ref{inversioneLaplaceGenerale})).
 
Note that in the case of $g$ being constant with respect to the variable $t$, the last term of (\ref{tesiGeneraleNonOmogenea}) reads $-g(x)t^{\nu N}E_{\nu, \nu N +1}^{(m_1,\dots,m_M)}\big(\eta t^\nu\big)$.
\end{remark}

\begin{remark}
Consider the real sequence $\{\nu_n\}_{n\in \mathbb{N}}$ such that $\nu_n\longrightarrow\nu>0$ and $\ceil{\nu_n N}=\ceil{\nu N}>0\ \forall\ n$. Then, 
\begin{equation}\label{risultatoLimite}
F_{\nu}(t,x) = \lim_{n\to\infty}F_{\nu_n}(t,x),\ \ t\ge0,\ x\in \mathbb{R}^d,
\end{equation}
where $F_{\nu},F_{\nu_n}$ are respectively the solutions to the problem of parameter $\nu$ and $\nu_n\ \forall\ n$, with the same initial conditions. This means that we can connect the limit of the solutions (pointwise) to the ``limit'' of the Cauchy problems (where the initial conditions stay the same because $\ceil{\nu_n N}=\ceil{\nu N}\ \forall\ n$).

Result (\ref{risultatoLimite}) comes from the continuity of the function (\ref{GenMittagLefflerMulti}) with respect to the fractional parameter $\nu>0$. This can be seen as a consequence of the continuity of the Gamma function on the real half-line and a suitable application of the dominated convergence theorem.
\end{remark}

\begin{theorem}\label{teoremaSubordinazione}
Let $d,N,n\in \mathbb{N},\ \nu>0$. Let $\lambda_0,\dots,\lambda_N\in \mathbb{R}$ and $\eta_1,\dots,\eta_M\in \mathbb{C}\setminus \{0\}$ satisfying condition (\ref{rappresentazionePolinomio}).
Then, the solution $F_{\nu/n}$ of the problem of parameter $\nu/n$
\begin{equation}\label{problemaOrdineFrazione}
\begin{cases}
\displaystyle\sum_{k=0}^N \lambda_k\frac{\partial^{\nu k/n} }{\partial t^{\nu k/n}}F(t, x) = 0,\ \ t\ge0,\ x\in \mathbb{R},\\[15pt]
\displaystyle\frac{\partial^{l} F}{\partial t^{l}}\Big|_{t=0} = f_l(x),\ \ x\in \mathbb{R}^d,\ l=0,\dots,\ceil{\frac{N\nu}{n}} -1,
\end{cases}
\end{equation}
can be expressed as
\begin{equation}\label{subordinazioneGenerale}
F_{\nu/n}(t,x) = \mathbb{E}\,F_{\nu}\Biggl(\,\prod_{j=1}^{n-1}G_{j}^{(n)}(t),\,x\Biggr),
\end{equation}
where the $G_j^{(n)}(t)$ are the random variables introduced in Section \ref{sottosezioneVariabiliAleatorieG} and $F_{\nu}$ is the solution to a problem of parameter $\nu$ with suitable initial condition
\begin{equation}\label{problemaOrdineInteroAssociato}
\begin{cases}
\displaystyle\sum_{k=0}^N \lambda_k\frac{\partial^{\nu k} }{\partial t^{\nu k}}F(t, x) = 0,\ \ t\ge0,\ x\in \mathbb{R}\\[15pt]
\displaystyle\frac{\partial^{l} F}{\partial t^{l}}\Big|_{t=0} = 
\begin{cases}
f_l,\ \ l=hn, \ \text{with } h =0,\dots,\ceil{N\nu/n} -1,\\
0,\ \ otherwise.
\end{cases}
\end{cases}
\end{equation}
\end{theorem}
Note that the conditions of the problem (\ref{problemaOrdineFrazione})  of degree $\nu/n$ appear in the associated problem (\ref{problemaOrdineInteroAssociato}) in the derivative whose order is multiple of $n$, while the other initial conditions are assumed equal to $0$. We also point out that all the conditions of the original problem always appear in the related problem since $n\Bigl(\ceil{\nu N/n}-1\Bigr)\le \ceil{\nu N}-1$.

\begin{proof}
We begin by showing a possible way to express the multivariate Mittag-Leffler (\ref{GenMittagLefflerMulti}) of fractional order $\nu/n$ in terms of that of fractional order $\nu$. Remember that for the gamma function, with $z\in\mathbb{C}$ and $n\in\mathbb{N}$ we can write (thanks to the $n$-multiplication formula of the Gamma function)
\begin{align}
\Gamma\Bigl(z+\frac{n-1}{n}\Bigr)^{-1} &= \frac{\prod_{j=1}^{n-1}\Gamma\Bigl(z+\frac{j-1}{n}\Bigr)}{(2\pi)^{\frac{n-1}{2}}n^{\frac{1}{2}-nz}\Gamma(nz)} \nonumber\\
&= \frac{1}{(2\pi)^{\frac{n-1}{2}}n^{\frac{1}{2}-nz}\Gamma(nz)}\prod_{j=1}^{n-1} \int_0^\infty e^{-w_j} w_j^{z+\frac{j-1}{n}-1}\dif w_j.\label{espressioneFunzioneGamma}
\end{align}
 Let $x\in\mathbb{C}^M$ and $L,h>0$,
\begin{align}
E_{\frac{\nu}{n},\, \frac{\nu}{n}L+ h}^{(m_1,\dots,m_M)}(x) &= \sum_{k_1,\dots,k_M=0} \Biggl(	\prod_{j=1}^M \frac{\Gamma(m_j+k_j)}{\Gamma(m_j)\,k_j!}x_j^{k_j}\Biggr) \Gamma\Biggl(\frac{\nu}{n}\sum_{j=1}^M k_j+\frac{\nu}{n}L+h\Biggr)^{-1}  \label{primoPassaggioSubordinazioneGenerale}\\
& =  \sum_{k_1,\dots,k_M=0} \Biggl( \prod_{j=1}^M \frac{\Gamma(m_j+k_j)}{\Gamma(m_j)\,k_j!}x_j^{k_j}\Biggr) \frac{\Gamma\Bigl(\nu\sum_{j=1}^M k_j+\nu L+ n h - (n-1)\Bigr)^{-1}}{(2\pi)^{\frac{n-1}{2}}n^{\frac{1}{2}-\bigl(\nu\sum_{h=1}^M k_h+\nu L+n h -(n-1)\bigr)}}\nonumber\\
&\ \ \ \times\prod_{j=1}^{n-1} \int_0^\infty e^{-w_j} w_j^{\frac{1}{n}\bigl(\nu\sum_{h=1}^M k_h+\nu L +n h-n+j\bigr)-1}\dif w_j \nonumber\\
& = \frac{n^{\nu L+n(h-1)+1/2}}{(2\pi)^{\frac{n-1}{2}}}\int_0^\infty \cdots\int_0^\infty \prod_{j=1}^{n-1} e^{-w_j} w_j^{\frac{1}{n}\bigl(\nu L+n h -n+j\bigr)-1}\dif w_j \nonumber \\
& \ \ \ \times E_{\nu,\, \nu L+n(h-1)+1}^{(m_1,\dots,m_M)}\Bigl(x\,n^\nu \prod_{j=1}^{n-1} w_j^{\nu/n}\Bigr)\label{rappresentazioneMitLefGen}
\end{align}
where in (\ref{primoPassaggioSubordinazioneGenerale}) we used (\ref{espressioneFunzioneGamma}) with $z = \frac{\nu}{n}\sum_{j=1}^M k_j+ h+\frac{\nu}{n}L -\frac{n-1}{n}$.

Now we apply (\ref{rappresentazioneMitLefGen}) (with $h=l+1$ and $L= N-k$) to formula (\ref{tesiGenerale}) and derive result (\ref{subordinazioneGenerale}). Let us consider $\eta = (\eta_1,\dots,\eta_M)$ given in the hypotheses, then
\begin{align}
F_{\nu/n}(t,x)& = \sum_{l=0}^{\ceil{\frac{\nu N}{n}}-1} f_{l}(x) \sum_{k=k_{l}}^N \lambda_k\, t^{\frac{\nu}{n}(N-k)+l} \, E_{\frac{\nu}{n}, \,\frac{\nu}{n}(N-k)+l+1}^{(m_1,\dots,m_M)}\Big(\eta_1t^{\nu/n}, \dots, \eta_Mt^{\nu/n}\Big)\nonumber\\
& = \sum_{l=0}^{\ceil{\frac{\nu N}{n}}-1} f_{l}(x) \sum_{k=k_{l}}^N \lambda_k\, t^{\frac{\nu}{n}(N-k)+l}\,\frac{n^{\nu (N-k)+nl+1/2}}{(2\pi)^{\frac{n-1}{2}}}\int_0^\infty \cdots\int_0^\infty \Biggl(\prod_{j=1}^{n-1} \dif w_j\Biggr) \nonumber \\
& \ \ \ \times \Biggl(\prod_{j=1}^{n-1} e^{-w_j} \Biggr) \Biggl(\prod_{j=1}^{n-1} w_j^{\frac{1}{n}\bigl(\nu (N-k)+n l -n+j\bigr)-1}\Biggr) E_{\nu,\, \nu (N-k)+nl+1}^{(m_1,\dots,m_M)}\Biggl(\eta\Bigl(nt^{1/n} \prod_{j=1}^{n-1} w_j^{1/n}\Bigr)^\nu\Biggr)\nonumber\\
& = \Bigl(\frac{n}{2\pi}\Bigr)^{\frac{n-1}{2}} \frac{1}{\sqrt{t}}\int_0^\infty \cdots\int_0^\infty \Biggl(\prod_{j=1}^{n-1} \dif  y_j\Biggr)\Biggl(\prod_{j=1}^{n-1}  y_j^{j-1}\Biggr)\Biggl(\prod_{j=1}^{n-1} e^{-\frac{y_j^n}{( n^nt)^{\frac{1}{n-1}}}} \Biggr)\nonumber\\
&\ \ \ \times  \sum_{l=0}^{\ceil{\frac{\nu N}{n}}-1} f_{l}(x) \sum_{k=k_{l}}^N \lambda_k \Biggl( \prod_{j=1}^{n-1} y_j\Biggr)^{\nu (N-k)+nl} E_{\nu,\, \nu (N-k)+nl+1}^{(m_1,\dots,m_M)}\Biggl(\eta \prod_{j=1}^{n-1} y_j^\nu\Biggr) \label{rappresentazioneEsplicitaSoluzioneGradoFrazione}
\end{align}
where in the last step we used the change of variables
$$ nt^{1/n} \prod_{j=1}^{n-1} w_j^{1/n} = \prod_{j=1}^{n-1} y_j \iff w_j =\frac{y_j^n}{\bigl( n^nt\bigr)^{\frac{1}{n-1}}}, \ \forall\ j \implies \prod_{j=1}^{n-1} \dif w_j = \frac{\prod_{j=1}^{n-1} \dif  y_j\, y_j^{n-1}}{nt}$$
and we performed some simplifications.
\\At last, we show that the second line of (\ref{rappresentazioneEsplicitaSoluzioneGradoFrazione}) coincides with the time-changed solution $F_{\nu}\Bigl(\,\prod_{j=1}^{n-1}G_{j}^{(n)}(t),\,x\Bigr)$ of the associated problem (\ref{problemaOrdineInteroAssociato}). Let us denote with $\tilde{f}_l$ the function appearing in the $l$-th condition of problem (\ref{problemaOrdineInteroAssociato}) and with $\tilde{k}_l = \min\{k=1,\dots,N\,:\,\nu k> l\}$ for $l = 0,\dots, \ceil{\nu N}-1$. Then, the solution of the related Cauchy problem reads
\begin{equation*}\label{FinteraParte1}
F_\nu(s,x) = \sum_{l=0}^{\ceil{\nu N}-1} \tilde{f}_{l}(x) \sum_{k=\tilde{k}_{l}}^N \lambda_k\, s^{\nu(N-k)+l} \, E_{\nu, \,\nu(N-k)+l+1}^{(m_1,\dots,m_M)}\Big(\eta_1t^\nu, \dots, \eta_Ms^\nu\Big),
\end{equation*}
where the functions $\tilde{f}_{l}$ are identically null for $l \not = nh$ with $h = 0,\dots, \ceil{\nu N/n}-1$, therefore we can write (removing the indexes of the null terms and performing the change of variable $l = nh$)
\begin{equation*}\label{FinteraParte2}
F_\nu(s,x) = \sum_{h=0}^{\ceil{\frac{\nu N}{n}}-1} \tilde{f}_{nh}(x) \sum_{k=\tilde{k}_{nh}+1}^N \lambda_k\, s^{\nu(N-k)+nh} \, E_{\nu, \,\nu(N-k)+nh+1}^{(m_1,\dots,m_M)}\Big(\eta_1s^\nu, \dots, \eta_Ms^\nu\Big).
\end{equation*}
By observing that $\tilde{k}_{nh} = \min\{k=1,\dots,N: \nu k> nh\} = \min\{k=1,\dots,N: \nu k/n> h\} =k_{h} \ \forall \ h$, we obtain the last line of (\ref{rappresentazioneEsplicitaSoluzioneGradoFrazione}) by setting $s = \prod_{j=1}^{n-1}y_j $.
\end{proof}

\begin{remark}[Brownian subordination]
If $n=2$, formula (\ref{subordinazioneGenerale}) becomes
\begin{equation}\label{subordinazioneGeneraleBrowniano}
F_{\nu/2}(t,x) = \mathbb{E}\,F_{\nu}\Bigl(\,|B(2t)|,\,x\Bigr),
\end{equation}
with $B$ standard Brownian motion (see Section \ref{sottosezioneVariabiliAleatorieG}).
\\Furthermore, by keeping in mind (\ref{subordinazioneGeneraleBrowniano}) and iterating the same argument, we obtain  that
\begin{equation}\label{subordinazioneGeneraleBrownianoIterato}
F_{\nu/2^{n}}(t,x) = \mathbb{E}\,F_{\nu}\Bigl(\,|B_n(2|B_{n-1}(2|\cdots 2|B_1(2t)| \cdots |\,)\,|\,)\,|,\,x\Bigr),
\end{equation}
where $B_1,\dots,B_n$ are independent standard Brownian motions and $F_\nu$ solution of the associated problem of the form (\ref{problemaOrdineInteroAssociato}) with $2^n$ replacing $n$.
\end{remark}

\subsection{Algebraic multiplicities equal to 1}

In this section we restrict ourselves to the case where the characteristic polynomial in (\ref{rappresentazionePolinomio}) has all distinct roots. This hypothesis permits us to present a more elegant result than that of Theorem \ref{teoremaGenerale}.

\begin{theorem}\label{teoremaMolteplicitaUnitarie}
Let $d,N\in \mathbb{N},\ \nu>0$ and $\lambda_0,\dots,\lambda_N\in \mathbb{R}$. If 
\begin{equation}\label{rappresentazionePolinomioMolteplicitaUnitarie}
\sum_{k=0}^N \lambda_k x^k = \prod_{j=1}^N(x-\eta_j) \ \ \text{with }\ \eta_1,\dots,\eta_N\in \mathbb{C}\setminus \{0\},
\end{equation}
then, the solution to the fractional Cauchy problem
\begin{equation}\label{problemaDifferenziale}
\begin{cases}
\displaystyle\sum_{k=0}^N \lambda_k \frac{\partial^{\nu k} }{\partial t^{\nu k}}F(t, x) = 0,\ \ t\ge0,\ x\in \mathbb{R}\\[15pt]
\displaystyle\frac{\partial^{l} F}{\partial t^{l}}\Big|_{t=0} = f_l(x),\ \ x\in \mathbb{R}^d,\ l=0,\dots,\ceil{\nu N} -1,
\end{cases}
\end{equation}
is the function $F:[0,\infty)\times\mathbb{R}^d\longrightarrow \mathbb{R}$ given by
\begin{equation}\label{soluzioneProblemaCasoMolteplicitaUnitarie}
F(t,x) = \sum_{h=1}^N\sum_{l=0}^{\ceil{\nu N}-1} E_{\nu, l+1}\bigl(\eta_h t^\nu\bigr) f_{l}(x) \,t^{l}  \sum_{k=k_{l}}^N \frac{\lambda_k \,\eta_h^{k-1}}{\displaystyle\prod_{\substack{j=1\\j\not=h}}^N(\eta_h-\eta_j)},
\end{equation}
with $k_l = \min\{k=1,\dots,N\,:\, \nu k> l\},\ l=0,\dots,\ceil{\nu N}-1$.
\end{theorem}

Note that in result (\ref{soluzioneProblemaCasoMolteplicitaUnitarie}) the fractional order $\nu$ influences only the fractional order of the Mittag-Leffler function (and the number of initial conditions), so the coefficients of the linear combination are constant (with respect to $\nu$). We point out that the series in (\ref{soluzioneProblemaCasoMolteplicitaUnitarie}) can be inverted becoming $\sum_{k=1}^N \sum_{l=0}^{\ceil{\nu k}-1}$.

\begin{proof}
First we note that, for $n\in \mathbb{N}_0$ and $l\in\mathbb{C}$,
\begin{equation}\label{rappresentazioneMitLefSemplificata}
E_{\nu,\,n\nu+l}(x) = \frac{E_{\nu,l}(x)}{x^n}-\sum_{j=1}^n \frac{x^{-j}}{\Gamma\bigl((n-j)\nu + l\bigr)}.
\end{equation}

Now, we proceed as in the proof of Theorem \ref{teoremaGenerale} and we perform the $t$-Laplace transform of the equation in problem (\ref{problemaDifferenziale}). In this case, formula (\ref{trasformataLaplaceGenerale}) reads 
\begin{equation}\label{trasformataLaplaceGeneraleNonOmogenea}
G(\mu,x)  = \frac{\displaystyle\sum_{l=1}^{\ceil{\nu N}} f_{l-1}(x)  \sum_{k=k_{l-1}}^N \lambda_k\, \mu^{\nu k -l}}{\displaystyle\prod_{h=1}^N \bigl(\mu^{\nu}-\eta_h\bigr)}.
\end{equation}
We now invert the functions $\mu^{\nu k -l} /\prod_{h=1}^N \bigl(\mu^{\nu}-\eta_h\bigr)$ for $l=1,\dots, \ceil{\nu N}$ and $k = k_{l-1}+1, \dots, N$. We note that 
$$\nu k - l = \nu - l\frac{N-k+1}{N} + (k-1)\Bigl(\nu- \frac{l}{N}\Bigr)$$
and therefore we write
\begin{align}
\mathcal{L}^{-1}& \Biggl(\frac{\mu^{\nu k -l} }{\displaystyle\prod_{h=1}^N \bigl(\mu^{\nu}-\eta_h\bigr)}\Biggr)(t) \nonumber\\
&=\mathcal{L}^{-1}\Biggl(  \frac{\mu^{\nu-l(N-k+1)/N} }{\mu^{\nu}-\eta_1}\Biggr)(t)\,\prod_{h=2}^{k} \mathcal{L}^{-1}\Biggl( \frac{\mu^{\nu -l/N}}{\mu^{\nu}-\eta_h}\Biggr)(t)\, \prod_{h=k+1}^N\mathcal{L}^{-1}\Bigl(  \frac{1 }{\mu^{\nu}-\eta_h}\Bigr)(t) \nonumber\\
& = t^{l(N-k+1)/N-1}E_{\nu,l\frac{N-k+1}{N}}\bigl(\eta_1 t^\nu\bigr)\ast\Conv_{h=2}^{k} t^{l/N -1} E_{\nu,\,\frac{l }{N}}\bigl(\eta_ht^\nu\bigr) \ast\Conv_{h=k+1}^{N} t^{\nu-1}E_{\nu,\,\nu }\bigl(\eta_ht^\nu\bigr)\label{passaggioConvoluzioneMittagLeffler}\\
& = \sum_{h=1}^N \frac{t^{\nu(N-k)+l-1}\eta_h^{N-1}}{\prod_{\substack{j=1\\j\not=h}}^N (\eta_h-\eta_j)} E_{\nu, \nu(N-k) +l}\bigl(\eta_h t^\nu\bigr)\label{passaggioRappresentazioneMitLef}\\
& = \sum_{h=1}^N \frac{t^{l-1}\eta_h^{k-1}}{\prod_{\substack{j=1\\j\not=h}}^N (\eta_h-\eta_j)} E_{\nu, l}\bigl(\eta_h t^\nu\bigr) -  \sum_{i=1}^{N-k}\frac{t^{\nu(N-k-i)+l-1}}{\Gamma\bigl(\nu(N-k-i)+l-1\bigr)} \sum_{h=1}^N \frac{\eta_h^{N-1-i}}{\prod_{\substack{j=1\\j\not=h}}^N (\eta_h-\eta_j)}\label{passaggioSemplificazioneEccezionale}\\
& =  \sum_{h=1}^N \frac{t^{l-1}\eta_h^{k-1}}{\prod_{\substack{j=1\\j\not=h}}^N (\eta_h-\eta_j)} E_{\nu, l}\bigl(\eta_h t^\nu\bigr), \label{inversioneFattoreCasoMolteplicitaUnitarie}
\end{align}
where in step (\ref{passaggioConvoluzioneMittagLeffler}) we used Proposition \ref{proposizioneConvoluzioneMittagLeffler}, in step (\ref{passaggioRappresentazioneMitLef}) we used (\ref{rappresentazioneMitLefSemplificata})  and changed the order of the sums in the second term, and in step (\ref{passaggioSemplificazioneEccezionale}) we used formula (\ref{formulaPerCasoMolteplicitaUnitarie}) (note that $N-i-1\le N- 2$ for each $i = 1,\dots,N-k$).
Finally, with formula (\ref{inversioneFattoreCasoMolteplicitaUnitarie}) at hand, the inversion of (\ref{trasformataLaplaceGeneraleNonOmogenea}) yields the claimed result (\ref{soluzioneProblemaCasoMolteplicitaUnitarie}) (after the change of variable $l'=l-1$).
\end{proof}

We observe that in the case where all the initial conditions are equal to null functions, except the first one, result (\ref{soluzioneProblemaCasoMolteplicitaUnitarie}) simplifies into
\begin{equation}\label{soluzioneProblemaCasoMolteplicitaUnitarieCondizioniNulle}
F(t,x) = \sum_{h=1}^N E_{\nu, 1}\bigl(\eta_h t^\nu\bigr) f_{0}(x)  \sum_{k=1}^N \frac{\lambda_k \,\eta_h^{k-1}}{\prod_{\substack{j=1\\j\not=h}}^N(\eta_h-\eta_j)}.
\end{equation}

\begin{remark}[Integer derivatives]
From Theorem \ref{teoremaMolteplicitaUnitarie}, by setting $\nu=1$ in (\ref{problemaDifferenziale}), we obtain the general solution to the integer order differential Cauchy problem. In particular, under the condition (\ref{rappresentazionePolinomioMolteplicitaUnitarie}), we can write
\begin{equation}\label{soluzioneProblemaInteroCasoMolteplicitaUnitarie}
F(t,x) = \sum_{h=1}^N e^{\eta_h t}\sum_{l=0}^{N-1}f_{l}(x) t^l\sum_{k=l+1}^N \frac{\lambda_k \,\eta_h^{k-1-l}}{\displaystyle\prod_{\substack{j=1\\j\not=h}}^N(\eta_h-\eta_j)}.
\end{equation}
Note that in this case $k_l=l+1\ \forall\ l$. Furthermore, for $l\ge1$, we can write 
\begin{equation}\label{MittagLefflerEsponenziale}
E_{1, l+1}(x) =\frac{1}{x^{l}}\Biggl(e^x- \sum_{i=0}^{l-1}\frac{x^i}{i!}\Biggr).
\end{equation}
In light of (\ref{MittagLefflerEsponenziale}), formula (\ref{soluzioneProblemaCasoMolteplicitaUnitarie}) can be written as
\begin{align}
F(t,x) &=  \sum_{h=1}^N\sum_{l=0}^{N-1} \Biggl(e^{\eta_h t^\nu}- \sum_{i=0}^{l-1}\frac{(\eta_h t)^i}{i!}\Biggr)\frac{f_{l}(x) t^l}{\eta_h^{l}} \sum_{k=k_l}^N \frac{\lambda_k \,\eta_h^{k-1}}{\displaystyle\prod_{\substack{j=1\\j\not=h}}^N(\eta_h-\eta_j)} \nonumber\\
& = \sum_{h=1}^N e^{\eta_h t^\nu}\sum_{l=0}^{N-1}  f_{l}(x) t^l  \sum_{k=l+1}^N \frac{\lambda_k \,\eta_h^{k-l-1}}{\displaystyle\prod_{\substack{j=1\\j\not=h}}^N(\eta_h-\eta_j)} - \sum_{l=0}^{N-1}f_{l}(x)\sum_{i=0}^{l-1}\frac{t^{i+l}}{i!}\sum_{k=l+1}^N \lambda_k \sum_{h=1}^N\frac{\eta_h^{k-l-1+i}}{\displaystyle\prod_{\substack{j=1\\j\not=h}}^N(\eta_h-\eta_j)}\nonumber
\end{align}
and the last term is equal to $0$ because the last sum is always null thanks to formula (\ref{formulaPerCasoMolteplicitaUnitarie}) (in fact, $k-l-1+i\le k-l-1+(l-1)\le k-2\le N-2$ ).

Finally, we observe that in the case of null initial conditions, except the first one, formula (\ref{soluzioneProblemaInteroCasoMolteplicitaUnitarie}) coincides with the solution (\ref{soluzioneProblemaCasoMolteplicitaUnitarieCondizioniNulle}) (where $\nu>0$) with the exponential function replacing the Mittag-Leffler function.
\end{remark}

\begin{remark}
We point out that the result in Theorem \ref{teoremaSubordinazione} can be directly proved also from formula (\ref{soluzioneProblemaCasoMolteplicitaUnitarie}). In particular, the case with $\nu/n = 1/n$ follows by suitably applying the following representation of the Mittag-Leffler function, with $h\in \mathbb{N}$,
\begin{align*}
E_{1/n, h}(x) &= \sqrt{\frac{n}{(2\pi)^{n-1}}} \frac{1}{x^{n(h-1)}}\int_0^\infty \cdots \int_0^\infty  \Biggl(\prod_{j=1}^{n-1}e^{-y_j} y_j^{j/n-1} \dif y_j \Biggr) \Biggl( e^{nx\bigl(\prod_{j=1}^{n-1}y_j\bigr)^{1/n}} \\
&\ \ \ - \sum_{i=0}^{n(h-1)-1} \Bigl(nx\prod_{j=1}^{n-1}y_j^{1/n}\Bigr)^i\frac{1}{i!}\Biggr),
\end{align*}
which in the case of $n=2$, after the change of variable $y_1= y^2$, can be written as
$$ E_{1/2, h}(x) =  \frac{2x^{2(1-h)}}{\sqrt{\pi}}\int_0^\infty e^{-y^2}\Biggl( e^{2xy} - \sum_{i=0}^{2h-3}\frac{ (2xy)^i}{i!}\Biggr) \dif y.$$
The above formulas can be derived as formula (2.9) of \cite{BO2003}.
\end{remark}

\section{Application to random motions with finite velocity}

Let $\bigl(\Omega, \mathcal{F},\{\mathcal{F}_t\}_{t\ge0}, P\bigr)$ be a filtered probability space and $d \in \mathbb{N}$. In the following we assume that every random object is suitably defined on the above probability space (i.e. if we introduce a stochastic process, this is adapted to the given filtration). 

Let $N$ be a homogeneous Poisson process with rate $\lambda>0$ and consider the vectors $v_0,\dots, v_M\in\mathbb{R}^d$. Let $V$ be a stochastic process taking values in $\{v_0,\dots,v_M\}\ a.s.$ and such that, for $t\ge0$,
$$p_k = P\{V(0)=v_k\},\ \  P\{V(t+\dif t) = v_k\,|\,V(t) = v_h,\, N(t,t+\dif t]=1\} = p_{hk},\ \ \ h,k = 0,\dots,M.$$
We say that $V$ is the process describing the velocity of the associated random motion $X$ defined as
\begin{equation}\label{definizioneMoto}
X(t) = \int_0^t V(s)\dif s = \sum_{i=0}^{N(t)-1} \bigl(T_{i+1} -T_i\bigr) V(T_i) + \bigl(t-T_{N(t)}\bigr) V(T_{N(t)}), \ \ t\ge0,
\end{equation}
where $T_i$ denotes the $i$-th arrival time of $N$ and $V(T_i)$ denotes the random speed after the $i$-th event recorded by $N$, therefore after the potential switch occurring at time $T_i$. The stochastic process $X$ describes the position of a particle moving in a $d$-dimensional (real) space with velocities $v_0, \dots, v_M$ and which can change its current velocity only when the process $N$ records a new event.

\subsection{Initial conditions for the characteristic function}

Denote with $\gamma_{hk}^{(t)}:[0,1]\longrightarrow \mathbb{R}^d$ the segment between $v_h t$ and $v_k t$, that is $\gamma_{hk}^{(t)}(\delta) = v_ht\delta + v_kt(1-\delta),\ \delta\in [0,1]$. Now, it is easy to see that the distribution of the position $X(t)$, conditionally on the occurrence of one Poisson event in $[0,t]$ and the two different velocities taken at time $0$ (say $v_h$) and $t$ (say $v_k$), is uniformly distributed on the segment between the two velocities at time $t$ (that is $\gamma_{hk}^{(t)}$). In formulas, for $h\not=k=0,\dots,M$,
\begin{equation}
P\{X(t)\in \dif x\,|\,V(0) =v_h, V(t) = v_k, N[0,t]=1\} = \frac{\dif x}{||v_h-v_k||t},\ \ \text{with }x\in \gamma_{hk}^{(t)}.
\end{equation}

Then, for $t\ge0$ in the neighborhood of $0$ we observe that there can occur at maximum one Poisson event and therefore not more than one change of velocity for the motion $X$. Thus, we can write the Fourier transform of the distribution of $X(t)$, for $\alpha \in \mathbb{R}^d$ (with $<\cdot,\cdot>$ denoting the dot product in $\mathbb{R}^d$),
\begin{align}
\mathbb{E} e^{i\,<\alpha, X(t)>} &= \bigl(1-\lambda t\bigr)\sum_{k=0}^M p_k\, e^{i\,<\alpha,v_kt>}+ \lambda t\sum_{k=0}^M p_k \,p_{kk}e^{i\,<\alpha,v_kt>}\nonumber\\
&\ \ \  +\lambda t\sum_{\substack{h,k=0\\h\not=k}}^M p_h\, p_{hk} \int_{\gamma_{hk}^{(t)}}\frac{e^{i\,<\alpha,x>}}{||v_h-v_k||t}\dif x\nonumber\\
& = \bigl(1-\lambda t\bigr)\sum_{k=0}^M p_k e^{it\,<\alpha,v_k>}+ \lambda t \sum_{h,k=0}^M p_h\, p_{hk}\int_0^1 e^{it\,<\alpha,\,v_h\delta + v_k(1-\delta)>}\dif \delta.\label{trasformataFourierIntorno0}
\end{align}
By means of ($\ref{trasformataFourierIntorno0})$ we easily derive the values of the derivatives of the Fourier transform of the distribution of the position $X(t)$ in the neighborhood of $0$ and therefore also in $t=0$, which will be used as initial conditions for the Cauchy problem.
We point out that function (\ref{trasformataFourierIntorno0}) is based on the first order approximation of the probability mass of the Poisson process in the neighborhood of $t=0$. However, this approximation is sufficient to provide the characteristic function in the neighborhood of $0$; in fact, the probability law of random motions with finite velocities are derived by requiring only the knowledge at the first order, therefore we do not need a further expansion to obtain the higher order derivatives (in $t=0$).
\\

In detail we obtain, with $n\in \mathbb{N}_0, \ \alpha\in \mathbb{R}^d$, for $t$ sufficiently close to $0$,
\begin{align}
&\frac{\partial^{n} }{\partial t^{n}}\mathbb{E} e^{i\,<\alpha, X(t)>}\nonumber\\
&= \sum_{k=0}^M p_k e^{it<\alpha, v_k>} \Bigl[-n\lambda + (1-\lambda t)i<\alpha,v_k>\Bigr] \bigl(i<\alpha,v_k>\bigr)^{n-1}\nonumber\\
&\ \ \  +\lambda\sum_{h,k=0}^M p_h\, p_{hk}\int_0^1 e^{it\,<\alpha,\,v_h\delta + v_k(1-\delta)>}\Bigl[n+it<\alpha,v_h\delta + v_k(1-\delta)>\Bigr]\bigl(i<\alpha,\,v_h\delta + v_k(1-\delta)>\bigr)^{n-1}\dif \delta,
\end{align}
which in $t= 0$ simplifies into
\begin{align}
\frac{\partial^{n} }{\partial t^{n}}\mathbb{E} e^{i\,<\alpha, X(t)>}\Big|_{t=0}&= \sum_{k=0}^M p_k \Bigl[-n\lambda + i<\alpha,v_k>\Bigr] \bigl(i<\alpha,v_k>\bigr)^{n-1}\nonumber\\
&\ \ \ +n\lambda\sum_{h,k=0}^M p_h\, p_{hk}\int_0^1\bigl(i<\alpha,\,v_h\delta + v_k(1-\delta)>\bigr)^{n-1}\dif \delta.\label{condizioneInizialeDerivataNesima}
\end{align}
For derivatives of order $0,1,2$ we can write, with $\alpha\in \mathbb{R}^d$,
\begin{align}
\mathbb{E} e^{i\,<\alpha, X(0)>} =1 ,\ \ \  &\label{derivataZeroMotoAleatorio}\\
\frac{\partial }{\partial t}\mathbb{E} e^{i\,<\alpha, X(t)>}\Big|_{t=0}\ &= i<\alpha,\sum_{k=0}^Mp_k v_k>,\label{derivataPrimaMotoAleatorio}\\
\frac{\partial^2 }{\partial t^2}\mathbb{E} e^{i\,<\alpha, X(t)>}\Big|_{t=0} &\nonumber\label{derivataSecondaMotoAleatorio}\\
= -2\lambda i&<\alpha,\sum_{k=0}^Mp_k v_k>- \sum_{k=0}^M p_k <\alpha,v_k>^2 +\lambda i<\alpha,\sum_{h,k=0}^Mp_hp_{hk} (v_h + v_k)>.
\end{align}
Formula (\ref{derivataZeroMotoAleatorio}) is due to the fact that the particle performing the random motion is always assumed to be in the origin of $\mathbb{R}^d$ at time $t=0$. It is interesting to observe that the first derivative, given in (\ref{derivataPrimaMotoAleatorio}), is equal to $0$ for all $\alpha\in\mathbb{R}^d$ if and only if $\sum_{k=0}^M p_k v_k=0$.

\begin{example}[Orthogonal planar random motion]

We consider a random motion $(X,Y)$ governed by a homogeneous Poisson process $N$ with rate $\lambda>0$, moving in the plane with the following orthogonal velocities,
\begin{equation}\label{direzioniMotoOrtogonale}
v_k=\Biggl(c\cos\Bigl(\frac{k\pi}{2}\Bigr), c\sin\Bigl(\frac{k\pi}{2}\Bigr)\Biggr),\ \ c>0  \text{ with } k=0,1,2,3,
\end{equation}
and such that from velocity $v_k$ the particle can uniformly switch either to $v_{k-1}$ or $v_{k+1}$, that is $P\{V(T_{n+1}) = v_{k+1} \,|\,V(T_n ) = v_k\} = P\{V(T_{n+1}) = v_{k-1} \,|\,V(T_n ) = v_k\} =1/2,\ k=0,1,2,3$. Therefore, the particle whose motion is described by $(X,Y)$ lies in the square $S_{ct} = \{(x,y)\in\mathbb{R}^2\,:\,|x|+|y|\le ct\}$ at time $t>0$ and at each Poisson event take a direction orthogonal to the current one (see Figure \ref{MPS_1}). We refer to \cite{CO2023} (and references herein) for further details on planar orthogonal random motions and \cite{CO2023b} for its three-dimensional version.

\begin{figure}
	\begin{minipage}{0.5\textwidth}
		\centering
		\begin{tikzpicture}[scale = 0.6]
		\draw[dashed, gray] (4,0) -- (0,4) node[above right, black, scale = 0.9]{$ct$};
		\draw[dashed, gray] (0,4) -- (-4,0) node[above left, black, scale = 0.9]{$-ct$};
		\draw[dashed, gray] (-4,0) -- (0,-4) node[below left, black, scale = 0.9]{$-ct$};
		\draw[dashed, gray] (0,-4) -- (4,0) node[above right, black, scale = 0.9]{$ct$};
		\draw[->, thick, gray] (-5,0) -- (5,0) node[below, scale = 1, black]{$\pmb{X(t)}$};
		\draw[->, thick, gray] (0,-5) -- (0,5) node[left, scale = 1, black]{ $\pmb{Y(t)}$};
		\draw (0,0)--(0.5,0)--(0.5,1)--(-0.3,1)--(-0.3,2)--(-1,2)--(-1,2.2);
		\filldraw (0,0) circle (0.8pt); \filldraw (0.5,0) circle (0.8pt); \filldraw (0.5,1) circle (1pt); \filldraw (-0.3,1) circle (0.8pt); \filldraw (-0.3,2) circle (0.8pt); \filldraw (-1,2) circle (0.8pt);
		\draw (0,0)--(0,-0.4)--(0.5,-0.4)--(0.5,-1.6)--(2.4,-1.6);
		\filldraw (0,0) circle (0.8pt); \filldraw(0,-0.4) circle (0.8pt); \filldraw (0.5,-0.4) circle (0.8pt); \filldraw (0.5,-1.6) circle(0.8pt);
		\draw (0,0)--(-0.8,0)--(-0.8,-2.2)--(-0.6,-2.2)--(-0.6,-2.6)--(-0.2,-2.6);
		\filldraw(-0.8,0) circle (0.8pt); \filldraw (-0.8,-2.2) circle (0.8pt); \filldraw(-0.6,-2.2) circle(0.8pt); \filldraw(-0.6,-2.6) circle (0.8pt);
		\draw (0,0)--(0,0.8)--(-1,0.8)--(-1,0.9)--(-1.4,0.9)--(-1.8,0.9)--(-1.8, -0.2);
		\filldraw(0,0.8)circle (0.8pt); \filldraw(-1,0.8)circle (0.8pt); \filldraw(-1,0.9)circle (0.8pt); \filldraw(-1.8,0.9)circle (0.8pt); 
		\end{tikzpicture}
		\caption{\small Sample paths of an orthogonal \newline planar motion.}\label{MPS_1}
	\end{minipage}\hfill
	\begin{minipage}{0.5\textwidth}
		\centering
		\begin{tikzpicture}[scale = 0.5]
		\draw[dashed, gray] (-3,5.196) -- (6,0) node[above right, black, scale = 0.9]{$ct$};
		\draw[dashed, gray] (6,0) -- (-3,-5.196) node[above left, black, scale = 0.9]{};
		\draw[dashed, gray] (-3,-5.196) -- (-3,5.196) node[below left, black, scale = 0.9]{};
		\draw (-3,0) node[above left, scale =1.1]{$-\frac{ct}{2}$};
		\draw (0,5.196) node[right, scale =1.1]{$\frac{\sqrt{3}ct}{2}$};
		\draw (0,-5.196) node[right, scale =1.1]{$-\frac{\sqrt{3}ct}{2}$};
		\draw[->, thick, gray] (-4.5,0) -- (7,0) node[below, scale = 1, black]{$\pmb{X(t)}$};
		\draw[->, thick, gray] (0,-5.5) -- (0,6) node[left, scale = 1, black]{ $\pmb{Y(t)}$};
		\draw (0,0)--(1.3,0)--(0.9,-0.693)--(-0.2,1.212)--(1, 1.212)--(1.3,1.212)--(1.2,1.0388);
		\filldraw (0,0) circle (1pt); \filldraw (1.3,0) circle (1pt); \filldraw (0.9,-0.693)circle (1pt); \filldraw (-0.2,1.212) circle (1pt); \filldraw (1,1.212) circle (1pt); \filldraw (1.3,1.212) circle (1pt); \filldraw (1.2,1.0388) circle (1pt);
		\draw (0,0)--(-1,1.732)--(-2.5,4.33)--(-1.5, 4.33);
		\filldraw (0,0) circle (1pt); \filldraw (-1,1.732) circle (1pt); \filldraw(-2.5,4.33) circle (1pt); \filldraw(-1.5, 4.33) circle (1pt);
		\draw (0,0)--(-0.75,-1.299)--(-0.25,-1.299)--(-1.25,-3.031)--(-0.45,-3.031)--(-1.05,-1.9918);
		\filldraw(-0.75,-1.299) circle (1pt); \filldraw (-0.25,-1.299) circle (1pt); \filldraw(-1.25,-3.031) circle(1pt); \filldraw(-0.45,-3.031)circle (1pt); \filldraw (-1.05,-1.9918) circle(1pt);
		\end{tikzpicture}
		\caption{\small Sample paths of a planar motion with three directions.}
		\label{MP3}
	\end{minipage}\hfil
\end{figure}

The probability distribution $p(x,y)\dif x \dif y = P\{X(t)\in\dif x, Y(t)\in\dif y\},\ t\ge0,\ x,y \in Q_{ct},$ of the position of the motion $(X,Y)$ satisfies the fourth-order differential equation
\begin{equation}\label{equazioneQuartoOrdine}
\Bigl(\frac{\partial^2 }{\partial t^2}+2\lambda\frac{\partial}{\partial t} +\lambda^2\Bigr)\biggl(\frac{\partial^2}{\partial t^2} +2\lambda\frac{\partial}{\partial t} -c^2\Bigl(\frac{\partial^2}{\partial x^2}+\frac{\partial^2}{\partial y^2}\Bigr)\biggr)p +c^4\frac{\partial^4 p}{\partial x^2\partial y^2} = 0,
\end{equation}
and it is known that the current position $\bigl(X(t),Y(t)\bigr)$ can be represented as a linear combination of two independent telegraph processes, In details, for $t\ge0$,
\begin{equation}\label{decomposizioneMotoOrtogonale}
\begin{cases}
X(t) = U(t) + V(t),\\
Y(t) = U(t) - V(t),
\end{cases}
\end{equation}
where $U=\{U(t)\}_{t\ge0}$ and $V=\{V(t)\}_{t\ge0}$ are independent one-dimensional telegraph processes moving with velocities $\pm c/2$ and with rate $\lambda/2$ (note that a similar results holds in the case of a non-homogeneous Poisson process as well, see \cite{CO2023}). 

The Fourier transforms of the equation (\ref{equazioneQuartoOrdine}) has the form
\begin{equation}\label{equazioneQuartoOrdineFourier}
\Bigl(\frac{\partial^2 }{\partial t^2}+2\lambda\frac{\partial}{\partial t} +\lambda^2\Bigr)\Bigl(\frac{\partial^2}{\partial t^2} +2\lambda\frac{\partial}{\partial t} +c^2(\alpha^2+\beta^2)\Bigr)F +c^4\alpha^2\beta^2F= 0,
\end{equation}
and by means of formulas 
(\ref{derivataZeroMotoAleatorio}), (\ref{derivataPrimaMotoAleatorio}), (\ref{derivataSecondaMotoAleatorio}) and (\ref{condizioneInizialeDerivataNesima}) the initial conditions are 
\begin{equation}\label{condizioniMotoPianoDirezioniOrtogonali}
F(0,\alpha,\beta) = 1,\ \ F_t(0,\alpha,\beta) =0,\ \ F_{tt}(0,\alpha,\beta) =-\frac{c^2}{2}\bigl(\alpha^2+\beta^2\bigr), \ \ F_{ttt}(0,\alpha,\beta) =\frac{\lambda c^2}{2}\bigl(\alpha^2+\beta^2\bigr).
\end{equation}
Now, the fractional version of equation (\ref{equazioneQuartoOrdineFourier}), written in the form (\ref{problemaDifferenziale}), with $\nu>0$, is
\begin{align}
\frac{\partial^{4\nu}F}{\partial t^{4\nu}} + 4\lambda \frac{\partial^{3\nu}F}{\partial t^{3\nu}}+\Bigl(5\lambda^2+c^2(\alpha^2+\beta^2)\Bigr)\frac{\partial^{2\nu}F}{\partial t^{2\nu}}&+ 2\lambda\Bigl(\lambda^2+c^2(\alpha^2+\beta^2)\Bigr)\frac{\partial^\nu F}{\partial t^\nu}\nonumber\\
&  + c^2\Bigl(\lambda^2(\alpha^2+\beta^2)+c^2\alpha^2\beta^2\Bigr) F= 0. \label{equazioneQuartoOrdineFourierFrazionaria}
\end{align}
Let $A = \sqrt{\lambda^2-c^2(\alpha^2-\beta^2)}$ and $B = \sqrt{\lambda^2-c^2(\alpha^2+\beta^2)}$. Note that $c^2(\alpha^2+\beta^2) = \lambda^2-\bigl(A^2+B^2\bigr)/2$. Then, the following equality holds  
\begin{align*}
x^4&+ 4\lambda x^3+\Bigl(5\lambda^2+c^2(\alpha^2+\beta^2)\Bigr)x^2+ 2\lambda\Bigl(\lambda^2+c^2(\alpha^2+\beta^2)\Bigr)x + c^2\Bigl(\lambda^2(\alpha^2+\beta^2)+c^2\alpha^2\beta^2\Bigr) x\\
& = \prod_{k=1}^4 (x-\eta_k),
\end{align*}
with
\begin{equation}\label{formulaEta}
\eta_1 = -\lambda-\frac{A+B}{2},\ \eta_2 = -\lambda+\frac{A-B}{2},\ \eta_2 = -\lambda-\frac{A-B}{2},\ \eta_4 = -\lambda+\frac{A+B}{2}.
\end{equation}
With this at hand, by means of Theorem \ref{teoremaMolteplicitaUnitarie} it is easy to calculate the solution to a fractional Cauchy problem associated with equation (\ref{equazioneQuartoOrdineFourierFrazionaria}).
\\For instance, in the case of initial conditions $F(0,\alpha,\beta) = 1$ and $\frac{\partial^l F}{\partial t^l}\big|_{t=0} = 0$ for all $l$ (whose values depend on $\nu$), the solution reads
\begin{align*}
 F_\nu&(t,\alpha, \beta)\\
 & = \Biggl(\lambda^2-\Bigl(\frac{A-B}{2}\Bigr)^2\Biggr)\Bigl(\lambda -\frac{A+B}{2} \Bigr) E_{\nu, 1}(\eta_1 t^\nu) + \Biggl(\lambda^2-\Bigl(\frac{A+B}{2}\Bigr)^2\Biggr)\Bigl(\lambda +\frac{A-B}{2} \Bigr) E_{\nu, 1}(\eta_2 t^\nu) \\
 & \ \ \ +\Biggl(\lambda^2-\Bigl(\frac{A+B}{2}\Bigr)^2\Biggr)\Bigl(\lambda -\frac{A-B}{2} \Bigr) E_{\nu, 1}(\eta_3 t^\nu) + \Biggl(\lambda^2-\Bigl(\frac{A-B}{2}\Bigr)^2\Biggr)\Bigl(\lambda +\frac{A+B}{2} \Bigr) E_{\nu, 1}(\eta_4 t^\nu)
 \end{align*}
with $\eta_i$ given in (\ref{formulaEta}).
\\In the case of initial conditions given by (\ref{condizioniMotoPianoDirezioniOrtogonali}) and $3/4<\nu\le1$ (so all the conditions are required), the solution reads
\begin{align*}
 F_\nu(t,\alpha, \beta) &= \frac{1}{4}\Biggl[ \Bigl(1 -\frac{\lambda}{A} \Bigr)\Bigl(1 -\frac{\lambda}{B} \Bigr) E_{\nu, 1}(\eta_1 t^\nu) +\Bigl(1 +\frac{\lambda}{A} \Bigr)\Bigl(1 -\frac{\lambda}{B} \Bigr)  E_{\nu, 1}(\eta_2 t^\nu) \\
 & \ \ \ +\Bigl(1 - \frac{\lambda}{A} \Bigr)\Bigl(1 + \frac{\lambda}{B} \Bigr)  E_{\nu, 1}(\eta_3 t^\nu) +\Bigl(1 +\frac{\lambda}{A} \Bigr)\Bigl(1 +\frac{\lambda}{B} \Bigr)  E_{\nu, 1}(\eta_4 t^\nu) \Biggr].
 \end{align*}
Note that for $\nu = 1$ this is the Fourier transform of the probability law of the orthogonal planar motion $(X,Y)$. This particular case can be also shown by considering the representation (\ref{decomposizioneMotoOrtogonale}) in terms of independent one-dimensional telegraph processes and their well-known Fourier transform (see for instance \cite{OB2004} formula (2.16)). 
\end{example}

\begin{example}[Planar motion with three directions]
Let us consider a planar random motion $(X,Y)$ governed by a homogeneous Poisson process with rate $\lambda>0$ and moving with velocities
\begin{equation}\label{velocitaMotoPianoLO}
v_0 = (c,0),v_1 = (-c/2,\sqrt{3}c/2),v_2 = (-c/2,-\sqrt{3}c/2), \ \ \text{ with } c>0.
\end{equation}
Let us assume that the particle starts moving with a uniformly chosen velocity among the three possible choices in (\ref{velocitaMotoPianoLO}) and at each Poisson event it uniformly selects the next one (including also the current one). This kind of motion are sometimes called as \textit{complete minimal planar random motion}, see \cite{CC2023, LO2004} for further details.
\\The support of the position of the stochastic dynamics at time $t\ge0$ is the triangle $T_{ct} =\{(x,y)\in\mathbb{R}^2\,:\,-ct/2\le x\le ct, (x-ct)/\sqrt{3}\le y\le (ct-x)/\sqrt{3} \}$  (see Figure \ref{MP3}). It is known that the probability distribution $p(x,y)\dif x \dif y = P\{X(t)\in\dif x, Y(t)\in\dif y\}$ of the position of the motion $(X,Y)$ satisfies the third-order differential equation
\begin{equation}\label{equazioneMotoPianoTreDirezioni}
\Bigl(\frac{\partial }{\partial t} + \frac{3\lambda}{2}\Bigr)^3 p -\frac{27\lambda^3}{8}+ \frac{27\lambda^2}{16} \frac{\partial p }{\partial t} -\frac{3}{4}c^2\Delta\Bigl(\frac{\partial }{\partial t} + \frac{3\lambda}{2}\Bigr)p-\frac{3}{4}c^3\frac{\partial^3 p}{\partial x\partial y^2}+\frac{c^3}{4}\frac{\partial^3 p}{\partial x^3}= 0,
\end{equation}
where $\Delta$ denotes the Laplacian operator. Note that equation (\ref{equazioneMotoPianoTreDirezioni}) can be derived by considering formula (1.9) of \cite{LO2004} and putting $3/2\lambda$ instead of $\lambda$ (this is sufficient by virtue of Remark 3.4 of \cite{CO2023}). 
\\The initial conditions of the Cauchy problem related to equation (\ref{equazioneMotoPianoTreDirezioni}) follow by suitably applying formulas (\ref{derivataZeroMotoAleatorio}), (\ref{derivataPrimaMotoAleatorio}) and (\ref{derivataSecondaMotoAleatorio}). In particular, for the Fourier transform of $p$, $F(t,\alpha,\beta) = \int_{\mathbb{R}^2} e^{i (\alpha x+ \beta y)}p(t,x,y)\dif x\dif y$, we derive
\begin{equation}\label{condizioniMotoPianoTreDirezioni}
F(0,\alpha,\beta) = 1,\ \ F_t(0,\alpha,\beta) =0,\ \ F_{tt}(0,\alpha,\beta) =-\frac{c^2}{2}\bigl(\alpha^2+\beta^2\bigr),
\end{equation}
obviously the first two conditions imply that $p(0,x,y) = \delta(x,y)$, with $\delta$ denoting the Dirac delta function centered in the origin, and $p_t(0,x,y) = 0\ \forall \ x,y$. We refer to \cite{LO2004} for further details about this motion, such as the explicit form of $p$ (see also \cite{CO2023b, Dc2002}).

Now, thanks to Theorem \ref{teoremaSubordinazione} we can easily give a probabilistic interpretation of the time-fractional version of order $\nu = 1/n, \ m\in\mathbb{N}$ of equation (\ref{equazioneMotoPianoTreDirezioni}), subject to the same initial conditions. Note that for $0<\nu\le1/3$ only the first condition is needed, for $1/3<\nu\le1/2$ the first two conditions are required and for $2/3<\nu\le1$ all three conditions are necessary.
\\In details, the fractional Cauchy problem after the Fourier transformation is given by
\begin{equation}\label{problemaFourierMotoPianoTreDirezioni}
\frac{\partial^3 F_\nu}{\partial t^3} +\frac{9\lambda}{2} \frac{\partial^2F_\nu }{\partial t^2} +\Biggl(\Bigl(\frac{3}{2}\Bigr)^4\lambda^2+\frac{3c^2(\alpha^2+\beta^2)}{4}\Biggr)\frac{\partial F_\nu}{\partial t} + \Bigl(\frac{9\lambda c^2(\alpha^2+\beta^2)}{8}+\frac{3ic^3\alpha\beta^2}{4}-\frac{ic^3\alpha^3}{4}\Bigr)F_\nu = 0
\end{equation}
subject to the initial conditions in (\ref{condizioniMotoPianoTreDirezioni}).
By means of Theorem \ref{teoremaSubordinazione} we have that the Fourier transform $F_{1/n}$ satisfying (\ref{problemaFourierMotoPianoTreDirezioni}) can be expressed as $ F_{1/n}(t,x) = \mathbb{E}\,F\Biggl(\,\prod_{j=1}^{n-1}G_{j}^{(n)}(t),\,x\Biggr),$ where $F$ denotes the solution to the Fourier-transformed problem of equation (\ref{equazioneMotoPianoTreDirezioni}). Thus, the time-fractional version of (\ref{equazioneMotoPianoTreDirezioni}), with $\nu = 1/n$ for natural $n$, describes the probability law of a stochastic process $(X_\nu,Y_\nu)$ that is a time-changed planar motion with velocities (\ref{velocitaMotoPianoLO}),
$$ \bigl(X_\nu(t),Y_\nu(t)\bigr) \stackrel{d}{=} \Biggl( X\Bigl(\prod_{j=1}^{n-1}G_{j}^{(n)}(t)\Bigr), Y\Bigl(\prod_{j=1}^{n-1}G_{j}^{(n)}(t)\Bigr) \Biggr),\ \ \ t\ge0.$$
\end{example}

\textbf{\large{Declarations}}
\\
\\
\textbf{Ethical Approval.} This declaration is not applicable.
 \\
\textbf{Competing interests.}  The authors have no competing interests to declare.
\\
\textbf{Authors' contributions.} Both authors equally contributed in the preparation and the writing of the paper.
 \\
\textbf{Funding.} The authors received no funding.
 \\
\textbf{Availability of data and materials.} This declaration is not applicable









\footnotesize{

}

\end{document}